\documentclass[12pt]{amsart}

\usepackage[margin=1in]{geometry}
\usepackage{amsmath,amssymb,amsthm,mathtools}
\usepackage[T1]{fontenc}
\usepackage{lmodern}
\usepackage{microtype}
\usepackage{enumitem}
\usepackage[colorlinks=true,linkcolor=blue,citecolor=blue,urlcolor=blue]{hyperref}

\newtheorem{theorem}{Theorem}[section]
\newtheorem{proposition}[theorem]{Proposition}
\newtheorem{lemma}[theorem]{Lemma}
\newtheorem{corollary}[theorem]{Corollary}
\theoremstyle{remark}
\newtheorem{remark}[theorem]{Remark}
\theoremstyle{definition}
\newtheorem{definition}[theorem]{Definition}

\newcommand{\DD}{\mathrm{DD}}
\newcommand{\DDh}{\overline{\DD}}
\newcommand{\DADP}{\mathrm{DADP}}

\newcommand{\DDP}{\mathrm{DDP}}
\newcommand{\Z}{\mathbb Z}
\newcommand{\R}{\mathbb R}

\title{The disjoint disks property for Busemann \(G\)-spaces}
\author[T. Fujioka]{Tadashi Fujioka}
\address[T. Fujioka]{Department of Applied Mathematics, Fukuoka University, Fukuoka 814-0180, Japan}
\email{tfujioka210@gmail.com}

\author[S. Gu]{Shijie Gu}
\address[S. Gu]{Department of Mathematics, Northeastern University, Shenyang, Liaoning, China, 110004}
\email{shijiegutop@gmail.com}
\dedicatory{In memory of Robert J. Daverman}
\date{\today}
\subjclass[2020]{Primary 53C70; Secondary 53C23, 57N15, 57P05, 54B15}
\keywords{Busemann \(G\)-space, Busemann conjecture, disjoint disks property,
disjoint homotopies property, generalized manifold, homotopical \(Z_2\)-set}

\begin{document}

\begin{abstract}
We prove that every finite-dimensional Busemann \(G\)-space of
dimension at least five has the disjoint disks property (DDP).
For a sufficiently small metric sphere \(L=S(c,r)\), we show that
every embedded arc contained in an exact distance level is a
homotopical \(Z_2\)-set in \(L\).  It follows that \(L\) has the
disjoint arc-disk property and the disjoint homotopies property.
Daverman's product theorem then gives DDP for \(L\times\mathbb R\),
and a local avoidance argument at the center yields DDP for the
ambient \(G\)-space. Since finite-dimensional Busemann \(G\)-spaces are generalized
manifolds, in dimensions at least five the remaining obstruction to
the Busemann conjecture is the resolution problem.
\end{abstract}

\maketitle

\section{Introduction}\label{sec:introduction}

Busemann's \emph{The Geometry of Geodesics} develops a metric analogue
of intrinsic differential geometry for spaces with no a priori analytic
structure.  The notion of a \(G\)-space provides the axiomatic framework for
this theory, abstracting properties of geodesics familiar
from Finsler geometry.  A fundamental part of Busemann's program is the
conjecture that finite-dimensional \(G\)-spaces are topological
manifolds; see~\cite[pp.~2, 49]{Busemann}.  Thus the manifold problem is
not ancillary to the theory: it is what would place the geometry of
\(G\)-spaces within the setting of manifold geometry.  The conjecture is known in dimensions at most four: the cases
\(n=1,2\) are due to Busemann~\cite[(9.7) and (10.4), pp.~46--47 and~52--53]{Busemann}, the
three-dimensional case to Krakus~\cite{Krakus}, and the
four-dimensional case to Thurston~\cite{Thurston}. 

A number of special cases of the Busemann conjecture are known under
additional curvature assumptions.  Berestovskii proved the conjecture for \(G\)-spaces with Alexandrov curvature locally bounded from below or from above
\cite{BerestovskiiLower,BerestovskiiUpper}.  Andreev proved it for
\(G\)-spaces of nonpositive curvature in the sense of Busemann
\cite{AndreevNPC}.  More recently, the authors used strainer maps to give an
alternative proof of Andreev's result and to establish a
topological stability theorem for locally Busemann nonpositively
curved \(G\)-spaces~\cite[Theorems~1.6 and~1.7]{FGBNPC}.
They also gave sufficient conditions for a Busemann \(G\)-space
to admit a \(C^1\) and DC atlas together with a continuous
Finsler metric compatible with its distance
\cite{FGFinsler}.  In another recent preprint~\cite{FGConvexBalls}, the authors proved that a Busemann \(G\)-space whose
sufficiently small metric balls are convex is a topological
manifold.  As a global consequence, a straight \(G\)-space
whose metric balls are convex is homeomorphic to Euclidean
space~\cite{FGConvexBalls}.

Thurston also showed that every finite-dimensional \(G\)-space
is a homology manifold \cite[Theorem~2.12]{Thurston}.  Together with the basic topological properties
of \(G\)-spaces established and collected by
Berestovskii--Halverson--Repov\v{s}~\cite{BHR}, this places every
finite-dimensional \(G\)-space in the class of generalized manifolds
that arises naturally in high-dimensional manifold recognition. The Busemann conjecture belongs to a broader circle of classical
manifold-recognition problems.  As emphasized by
Halverson--Repov\v{s}~\cite{HR}, it is a special case of the
Bing--Borsuk conjecture and is closely related to the generalized
R.~L.~Moore problem concerning codimension-one manifold factors.
These connections motivate the use of general-position methods in the
study of Busemann \(G\)-spaces.  In dimensions at least five, the
disjoint disks property is a fundamental general-position condition
in manifold recognition and decomposition theory. The Edwards--Quinn recognition theorem
states that a resolvable generalized \(n\)-manifold, \(n\geq5\), is a
topological manifold if and only if it has the disjoint disks property
(DDP)~\cite{Edwards,QuinnResolution, QuinnObstruction}; see also \cite{DavermanBook}.  The DDP does not imply resolvability: for every \(n\geq6\),
there exist nonresolvable generalized \(n\)-manifolds with the DDP
\cite{BFMWTopology,BFMWErratum,BFMWDesingularizing}.  Thus DDP alone does not settle the conjecture.

Halverson--Repov\v{s} explicitly posed whether
every \(n\)-dimensional \(G\)-space, \(n\geq5\), has DDP and suggested the
closely related problem of finding a general-position property of a small
metric sphere which forces its product with \(\mathbb R\) to have DDP
\cite[Section~6, Question~2]{HR}.  Andreev later recorded the same question
as Problem~12.4 in his survey \cite[Problem~12.4]{AndreevSurvey}.  Theorem
\ref{thm:main} gives an affirmative answer.

\begin{theorem}\label{thm:main}
Every finite-dimensional Busemann \(G\)-space of dimension \(n\geq5\)
has the disjoint disks property.
\end{theorem}

Combining Theorem~\ref{thm:main} with manifold recognition gives the
following reformulation of what remains of the Busemann conjecture in high
dimensions.

\begin{corollary}\label{cor:resolution}
Let \(X\) be a finite-dimensional Busemann \(G\)-space of dimension
\(n\geq5\).  Then \(X\) is a topological \(n\)-manifold if and only if
\(X\) is resolvable.
\end{corollary}

% \begin{proof}
% A finite-dimensional Busemann \(G\)-space is a generalized manifold
% \cite{Thurston,BHR}.  If it is resolvable, Theorem~\ref{thm:main} and the
% Edwards--Quinn recognition theorem imply that it is a manifold.  The converse
% is immediate, since a manifold is resolved by the identity map.
% \end{proof}

% \begin{remark}
% The theorem does not settle the Busemann conjecture.  It removes the DDP
% part of the recognition problem; the resolution obstruction remains.  In
% general, resolvability of homology manifolds is governed by Quinn's resolution
% obstruction \cite{QuinnObstruction}.
% \end{remark}

\subsection{Idea of the proof}

Let \(X\) be a finite-dimensional \(G\)-space of dimension \(n\geq5\). The proof is organized around a sufficiently small metric sphere
\(L=S(c,r)\) in \(X\).  Fix \(p\in L\) and set \(h(x)=d(p,x)\) on \(L\).
Lemma~\ref{lem:level-deformation} gives a small continuous motion
\(T(z,t)\) in \(L\) for which \(h(T(z,t))\) is strictly larger than
\(h(z)\) when \(t>0\), and strictly smaller when \(t<0\).
This level-changing motion is the main metric ingredient in the proof.
The construction essentially goes back to Krakus~\cite{Krakus} and
Thurston~\cite{Thurston}, and was also used by
Berestovskii--Halverson--Repov\v{s}~\cite{BHR}.

For \(0<s<2r\), let \(E_s=h^{-1}(s)\).  We show that the corresponding
closed sublevel and superlevel sets are contractible ANRs and homology manifolds with boundary \(E_s\). Mitchell's boundary
theorem \cite{Mitchell} identifies \(E_s\) as a cohomologically locally
connected homology \((n-2)\)-manifold in his sense.  Bredon's comparison and
nonseparation theorems \cite[Theorems~V.16.8 and~V.16.20]{Bredon}, together
with local path connectedness of \(E_s\), give arcwise connectedness of
complements of embedded arcs in connected open subsets of \(E_s\).
Combining this with the level-changing motion gives local path and loop
control in \(L\setminus A\) for an embedded arc \(A\subset E_s\).
The local-negligibility criterion of Banakh--Cauty--Karassev
\cite[Theorem 3.2(1)]{BCK} then implies that \(A\) is a
homotopical \(Z_2\)-set in \(L\).

Every path in \(L\) can be approximated by a path whose image is contained
in a finite union of such level arcs.  This gives the disjoint arc-disk
property (DADP) for \(L\).  Daverman's product theorem
\cite[Proposition~2.10]{Daverman} then gives DDP for \(L\times\mathbb R\).
As an additional consequence, \(L\) has the disjoint homotopies property
(DHP), by \cite[Theorem~8.3(2)]{BV}.

A punctured sufficiently small metric ball in \(X\) has radial product
coordinates with factor \(L\).  To pass through its center, we use the
simple connectivity of the link to show that the center is a homotopical
\(Z_2\)-point, again by Banakh--Cauty--Karassev.  Disk maps are first moved
off the center and then separated in a punctured annulus.  Finally,
Banakh--Valov's local-to-global result \cite[Proposition~5.4(2)]{BV} gives DDP for \(X\).

The dimension restriction enters in Bredon's nonseparation argument.  The
level \(E_s\) has dimension \(n-2\), while the exceptional set is an arc,
so one needs \(1\leq(n-2)-2\), which is exactly \(n\geq5\).

\subsection{Conventions and basic terminology}

We write
\[
 U(x,r)=\{y:d(x,y)<r\},\qquad B(x,r)=\{y:d(x,y)\leq r\}, \qquad S(x,r) = \{y:d(x,y)=r\}
\]
for the open metric ball, the closed metric ball and the metric sphere centered at \(x\) with radius \(r\), respectively.  Throughout, dimension means Lebesgue covering dimension.

We use Busemann's original meaning of a \(G\)-space.  Thus \(X\) is a metric
space satisfying the following four axioms:
\begin{enumerate}[label=(\roman*)]
\item if \(x\ne y\), there is \(z\notin\{x,y\}\) such that
      \(d(x,z)+d(z,y)=d(x,y)\);
\item every bounded infinite subset has an accumulation point;
\item for every \(w\in X\) there is \(\rho_w>0\) such that, for any
      distinct \(x,y\in U(w,\rho_w)\), there is
      \(z\in U(w,\rho_w)\setminus\{x,y\}\) with
      \(d(x,y)+d(y,z)=d(x,z)\);
\item if \(x\ne y\), \(d(x,y)+d(y,z_i)=d(x,z_i)\) for \(i=1,2\),
      and \(d(y,z_1)=d(y,z_2)\), then \(z_1=z_2\).
\end{enumerate}
See \cite[Definition~2.1]{BHR}.  A \emph{geodesic segment} is an isometric
map of a compact interval into \(X\); we also use this term for its image.
The axioms imply that \(X\) is a complete, locally compact geodesic space,
that sufficiently close points are joined by a unique segment, and that
these segments admit locally unique prolongations
\cite[Remark~2.2]{BHR}.  Thus (iii) is local extendability, and (iv) is
uniqueness of a prolongation with prescribed length.

We recall the topological terminology used below.  A metric space \(Y\) is an
\emph{absolute neighborhood retract} (ANR) if, whenever \(Y\) is embedded as
a closed subset of a metric space \(Z\), some neighborhood of \(Y\) in \(Z\)
retracts onto \(Y\).  It is an \emph{absolute retract} (AR) if \(Y\) is a
retract of every metric space in which it is embedded as a closed subset.  We shall use the standard
facts that a finite-dimensional locally contractible separable metric space is an ANR,
and that a contractible metric ANR is an AR; see
\cite{Hanner} and \cite[Theorem~2.7]{Thurston}.

Unless another theory is specified, homology means singular homology with
integer coefficients.  A locally compact Hausdorff space \(M\) is an \emph{integral homology
\(m\)-manifold} if, for every \(x\in M\),
\[
 H_i(M,M\setminus\{x\};\mathbb Z)\cong
 H_i(\mathbb R^m,\mathbb R^m\setminus\{0\};\mathbb Z)
\]
for all \(i\).  A locally compact (separable) metric ANR which is an
integral homology \(m\)-manifold is called a \emph{generalized
\(m\)-manifold}.  For the levels \(E_s\), we use the homology-manifold
notion in the Borel--Moore theory of Mitchell and Bredon
\cite{Mitchell,Bredon}; this will be indicated explicitly.  We do not assume
local contractibility of \(E_s\), or identify its homology in that theory
with singular homology.

A metric space \(Y\) has the \emph{disjoint disks property} (DDP) if, for any
two maps \(f,g\colon D^2\to Y\) and every \(\varepsilon>0\), there are
\(\varepsilon\)-close maps \(f',g'\) with
\[
                         f'(D^2)\cap g'(D^2)=\varnothing .
\]
It has the \emph{disjoint arc-disk property} (DADP) if the analogous
statement holds for maps of \(D^1\) and \(D^2\).  It has the
\emph{disjoint homotopies property} (DHP) if any two path homotopies
\(D^1\times I\to Y\) can be approximated so that their images at each common
time are disjoint.

We shall also use the controlled homotopical notation of Banakh--Valov
\cite[Definitions~3.1 and~5.1]{BV}.  If \(\mathcal U\) is an open cover of
\(Y\), two maps with the same domain are \(\mathcal U\)-homotopic if they
are joined by a homotopy for which each point-track is contained in a member
of \(\mathcal U\).  For nonnegative integers \(m,a,b\), the property
\(m\text{-}\DDh^{\{a,b\}}\) means that, for every open cover
\(\mathcal U\) of \(Y\) and maps
\[
 f\colon I^m\times I^a\to Y,\qquad g\colon I^m\times I^b\to Y,
\]
there are maps \(f',g'\), respectively \(\mathcal U\)-homotopic to \(f,g\),
such that
\[
 f'(\{z\}\times I^a)\cap g'(\{z\}\times I^b)=\varnothing
 \qquad\text{for every }z\in I^m.
\]
In particular, \(0\text{-}\DDh^{\{1,2\}}\),
\(1\text{-}\DDh^{\{1,1\}}\), and \(0\text{-}\DDh^{\{2,2\}}\) are the
homotopical forms of DADP, DHP, and DDP, respectively.  For DHP, the common
coordinate \(z\in I\) is the time parameter.  For metric ANRs these
controlled homotopical formulations are equivalent to the corresponding
approximation formulations: sufficiently close maps are joined by
arbitrarily small homotopies; see Hanner~\cite[Theorem~4.1]{Hanner}.

\section{Small metric spheres and radial operations}

We use the standard local consequences of the Busemann \(G\)-axioms;
see Busemann~\cite{Busemann}, Thurston~\cite{Thurston} and
Berestovskii--Halverson--Repov\v{s}~\cite{BHR}.  Throughout, \(X\) denotes a \(G\)-space of finite dimension \(n\).  Following \cite[Section~3]{BHR}, define the \emph{local
extendability radius} \(\rho(c)\in(0,\infty]\) to be the supremum of the
radii \(R>0\) for which axiom~(iii) holds with \(w=c\) and \(\rho_w=R\).
Choose once and for all
\[
                              0<8r<\rho(c).
\tag{2.1}\label{eq:radius-choice}
\]  This choice ensures that:
\begin{enumerate}[label=(G\arabic*),ref=G\arabic*]
\item \label{cond:G1} any two points used below lie in \(U(c,\rho(c))\) and are joined
by a unique shortest segment;
\item \label{cond:G2} these segments depend continuously on their endpoints;
\item \label{cond:G3} all prolongations of length less than \(r\) used below are available
and unique.
\end{enumerate}
In particular, every sufficiently small ball centered at \(c\)
contracts radially to \(c\).  Since \(c\) is arbitrary, \(X\) is
locally contractible and hence, being finite-dimensional, is an
ANR~\cite[Proposition~3.13]{BHR}.

Set \(L = S(c,r)\). Radial coordinates give a homeomorphism
\[
 U(c,4r)\setminus\{c\}\cong L\times(0,4r).
\tag{2.2}\label{eq:radial-product}
\]
Indeed, by the canonical geodesic-cone structure of
\(B(c,4r)\)~\cite[Proposition~3.3]{BHR}, every point of
\(U(c,4r)\setminus\{c\}\) lies on a unique radial geodesic from \(c\),
and this geodesic meets \(L=S(c,r)\) in a unique point.  Under the
homeomorphism above, the second coordinate is the distance from \(c\),
and the first is this intersection point with \(L\). 

% The local contractibility of \(L\) follows directly from the proof of
% \cite[Theorem~2.12]{Thurston}; see also \cite[Proposition~3.18]{BHR}.
% Indeed, for \(q\in L\) and sufficiently small \(\delta>0\), join each point
% of \(L\cap U(q,\delta)\) to \(q\) by its segment and project the segment
% radially from \(c\) back to \(L\).  Krakus' second retraction inequality
% \cite[Corollary~2.4]{Thurston} keeps the contraction in
% \(L\cap U(q,\delta)\).  Such neighborhoods form a basis at \(q\).
% Thus \(L\) is a compact finite-dimensional ANR, and therefore an ENR.
% Moreover, \(L\) is an integral homology \((n-1)\)-manifold with the
% homology of \(S^{n-1}\)~\cite[Theorem~3.19]{BHR}.  We shall also use that
% \(X\) itself is an ANR~\cite[Proposition~3.13]{BHR}.

\begin{lemma}\label{lem:link-pi1}
The sphere \(L\) is a compact ANR homology \((n-1)\)-manifold with
the homology of \(S^{n-1}\). If \(n\geq3\), then \(L\) is simply connected.
\end{lemma}

\begin{proof}
The homology-manifold and homology-sphere assertions are
\cite[Theorem~3.19]{BHR}.  Local contractibility follows from the
proof of \cite[Theorem~2.12]{Thurston}, or from
\cite[Proposition~3.18]{BHR}.  Since \(L\) is
finite-dimensional, it is an ANR.  Simple connectivity for \(n\geq3\)
is \cite[Theorem~3.22]{BHR}.
For \(n\geq3\), these properties also imply that \(L\) is homotopy equivalent
to \(S^{n-1}\); see \cite[Theorem~1.1]{GuSpheres}.
We shall not use this stronger conclusion.
\end{proof}

For a point \(a\) and a nearby point \(z\ne a\), write \(H_a^t(z)\)
for the point on the segment from \(a\) through \(z\), prolonged when
\(t>0\) and shortened when \(t<0\), such that
\[
d(a,H_a^t(z))=d(a,z)+t,\qquad
d(z,H_a^t(z))=|t|.
\tag{2.3}\label{eq:Ht}
\]
By (\ref{cond:G1})--(\ref{cond:G3}), these maps are well-defined 
and are jointly continuous on their natural common domains.

Fix a point \(p\in L\) and consider the distance function
\[
 h\colon L\longrightarrow[0,2r],\qquad h(x)=d(p,x).
\]
For \(0<s<2r\), set
\[
 C_s=h^{-1}([0,s]),\qquad
 D_s=h^{-1}([s,2r]),\qquad
 E_s=h^{-1}(s).
\]
Thus \(E_s=L\cap S(p,s)\).  We call \(C_s\), \(D_s\), and \(E_s\),
respectively, the closed sublevel set, the closed superlevel set, and
the level set of \(h\) at \(s\). The antipode \(p^*\in L\) is the
unique point satisfying \(d(p,p^*)=2r\).

The following construction is a two-sided form of Krakus' retraction
inequalities.  The
negative-time inequality is the second retraction inequality~\cite[Corollary~2.4]{Thurston}, while the
positive-time inequality follows from the same strict radial comparison
as the first \cite[Proposition~2.3]{Thurston}.  The radial contraction in \cite[Proposition~3.18]{BHR}
uses the same technique.  

\begin{lemma}[local level-changing motion]\label{lem:level-deformation}
For every compact interval \(J\subset(0,2r)\), there is
\(\varepsilon_J>0\) and a jointly continuous map
\[
 T\colon h^{-1}(J)\times(-\varepsilon_J,\varepsilon_J)\longrightarrow L,
 \qquad T(z,0)=z,
\]
such that, for \(t\neq0\),
\[
 \operatorname{sign}\bigl(h(T(z,t))-h(z)\bigr)
       =\operatorname{sign}t,
\tag{2.4}\label{eq:sign}
\]
and \(d(T(z,t),z)\leq2|t|\).
\end{lemma}

\begin{proof}
By (\ref{cond:G1})--(\ref{cond:G3}) and the compactness of
\(h^{-1}(J)\), choose \(\varepsilon_J>0\) so that the two
prolongations used below are defined whenever \(|t|<\varepsilon_J\).
Put \(y=H_p^t(z)\) and \(\delta=d(c,y)-r\), and return \(y\) to
\(L\) along the radial segment from \(c\) by setting
\(T(z,t)=H_c^{-\delta}(y)\).
Then \(T(z,t)\in L\).  The triangle inequality gives
\(|\delta|\leq |t|\).  If equality held for \(t\ne0\), then
\(c,z,y\) would lie on one geodesic, with the order determined by the
sign of \(t\).  Since \(p,z,y\) already lie on one geodesic,
uniqueness of prolongation would put \(c,p,z\) on a single geodesic.
Because \(d(c,p)=d(c,z)=r\), this would force \(z=p\) or
\(z=p^*\), contrary to \(J\subset(0,2r)\).  Therefore
\[
                              |\delta|<|t|.
\tag{2.5}\label{eq:delta}
\]
Now
\[
 \left|d(p,T(z,t))-d(p,y)\right|
 \leq d(T(z,t),y)=|\delta|<|t|,
\]
whereas \(d(p,y)=d(p,z)+t\).  This proves
\eqref{eq:sign}. The displacement estimate follows from \eqref{eq:Ht},
\eqref{eq:delta}, and the triangle inequality.  Joint continuity
follows from that of the radial maps
\((z,t)\mapsto H_p^t(z)\) and \((y,u)\mapsto H_c^u(y)\),
since \(\delta=d(c,H_p^t(z))-r\) is continuous in \((z,t)\).
\end{proof}

The strict change relative to the initial level in \eqref{eq:sign} is the
metric property used in Sections~3 and~4.  We do not need a flow law for
\(T\), or monotonicity of \(t\mapsto h(T(z,t))\).

\section{Sublevel and superlevel sets of the distance function}

We now use Lemma~\ref{lem:level-deformation} to analyze the topology of the
sublevel and superlevel sets of \(h\).  These sets are absolute retracts,
and Mitchell's boundary theorem identifies \(E_s\) as a homology manifold
of codimension one in \(L\).  We assume \(n\geq3\) in this section.

Most of the results in this section, especially those concerning \(C_s\),
can be found in Thurston~\cite{Thurston} and
Berestovskii--Halverson--Repov\v{s}~\cite{BHR}.  With Mitchell's
homology convention taken into account, the results concerning \(C_s\)
suffice to establish the homology-manifold structure and connectivity
properties of \(E_s\) needed in Section~\ref{sec:arcs-in-a-distance}.
Nevertheless, we give a symmetric argument for \(C_s\) and \(D_s\),
using the level-changing motion \(T\) of
Lemma~\ref{lem:level-deformation}.  Our purpose is not only to establish
these properties of \(E_s\), but also to illustrate the deformation
method used in the local arguments of
Section~\ref{sec:arcs-in-a-distance}.

\begin{lemma}\label{lem:filtration}
If \(0<a<b<2r\), the inclusions
\(
 C_a\hookrightarrow C_b\)
 and 
 \(D_b\hookrightarrow D_a
\)
are homotopy equivalences.  For the first inclusion the homotopies may be
chosen to preserve every sublevel set of \(h\); for the second they may
be chosen to preserve every superlevel set.
\end{lemma}

\begin{proof}
Choose \(0<a_0<a\).  By Lemma~\ref{lem:level-deformation},
choose a fixed negative time \(\tau<0\), sufficiently close to zero,
such that \(T\) is defined on
\(h^{-1}([a_0,b])\times[\tau,0]\).  Let \(\chi\colon[0,2r]\to[0,1]\) be continuous, zero on
\([0,a_0]\), and one on \([a,2r]\), and put
\[
                  F(z)=T\bigl(z,\tau\chi(h(z))\bigr),\qquad z\in C_b,
\]
with \(F(z)=z\) where \(\chi(h(z))=0\).  This is continuous and satisfies \(h(F(z))\leq h(z)\), so it preserves
both \(C_a\) and \(C_b\).  On the compact
set \(h^{-1}([a,b])\), strictness and compactness give
\[
 \eta=\min\{h(z)-h(T(z,\tau)):a\leq h(z)\leq b\}>0.
\]
Choose an integer \(N>(b-a)/\eta+1\).  Then \(F^N(C_b)\) is
contained in \(C_a\).  Scaling the time
\(\tau\chi(h(z))\) from zero to its final value gives a homotopy from
the identity to \(F\) which preserves both \(C_a\) and \(C_b\).  Concatenating the
homotopies for the \(N\) iterates shows that
\(F^N\colon C_b\to C_a\) is a homotopy inverse to the inclusion:
the same homotopy restricted to \(C_a\) stays in \(C_a\).
For \(D_b\hookrightarrow D_a\), choose \(b<b_0<2r\), use positive time
on \(h^{-1}([a,b_0])\), and take a cutoff equal to one on \([0,b]\) and
zero on \([b_0,2r]\).  The same argument, with a uniform positive increase
on \(h^{-1}([a,b])\), gives the required homotopy equivalence.  Each
intermediate map preserves every superlevel set.
\end{proof}

The assertions about \(C_s\) in the next proposition follow more directly
from the proof of \cite[Theorem~2.12]{Thurston}; compare
\cite[Proposition~3.18]{BHR}.  Here we give a symmetric proof for
\(C_s\) and \(D_s\).

\begin{proposition}\label{prop:subsuper-ar}
For every \(0<s<2r\), the sets \(C_s\) and \(D_s\) are compact absolute retracts.  If \(A\subset E_s\) is closed, then
\[
                              C_s\setminus A\simeq *,
                 \qquad       D_s\setminus A\simeq * .
\tag{3.1}\label{eq:subsuper-complements}
\]
\end{proposition}

\begin{proof}
We first prove contractibility. For \(C_s\),  
choose a neighborhood \(N\) of \(p\)
whose inclusion in \(\{h<s\}\) is null-homotopic.  This is possible
by local contractibility of \(L\).  For sufficiently small \(0<a<s\),
we have \(C_a\subset N\).  Lemma~\ref{lem:filtration} deforms
\(C_s\) into \(C_a\), and the chosen null-homotopy then contracts
it in \(\{h<s\}\).
The same proof applies to \(D_s\): choose a neighborhood of \(p^*\)
contracting in \(\{h>s\}\), take \(s<b<2r\) sufficiently close to
\(2r\) that \(D_b\) lies in this neighborhood, and use the positive-time
deformation of Lemma~\ref{lem:filtration}.

For a closed set \(A\subset E_s\), restrict these deformations to the
complements of \(A\).  In the sublevel case, points initially below
\(s\) remain below \(s\), while points of \(E_s\setminus A\) move
strictly below \(s\) at every positive time of the first deformation
step.  The remaining deformation and contraction stay below \(s\).
The superlevel case is identical with the inequalities reversed.
This proves \eqref{eq:subsuper-complements}.

It remains to check local contractibility at a point \(z\in E_s\); away from \(E_s\) it follows from local
contractibility of \(L\). Let \(U\) be a relative open neighborhood of \(z\) in \(C_s\).  Choose \(\tau<0\) sufficiently close to zero that
\[
                         T(z,u\tau)\in U\qquad(0\leq u\leq1).
\]
The point \(z_-=T(z,\tau)\) lies in \(\{h<s\}\).  Choose a neighborhood \(W_-\)
of \(z_-\) whose inclusion in \(U\cap\{h<s\}\) is null-homotopic.
By continuity, there is a relative neighborhood \(V\) of \(z\) in \(C_s\) such that
\[
 T(V,u\tau)\subset U,\qquad T(V,\tau)\subset W_-.
\]
Concatenating this homotopy with the chosen null-homotopy of
\(W_-\hookrightarrow U\cap\{h<s\}\) proves local contractibility of \(C_s\).  The argument for \(D_s\) is the same, using
positive time.  Thus both sets are locally contractible and hence, being
finite-dimensional, are ANRs; being contractible, they are ARs.
\end{proof}

The assertion about \(C_s\) in the next proposition is the argument of \cite[Proposition~3.8]{Thurston},
with Mitchell's homology convention made explicit.  Thurston takes
\(s\leq r\) in the setup of \cite[Theorem~2.12]{Thurston}; the contraction
and local-homology calculation used there apply throughout \(0<s<2r\).
Indeed, for contraction toward \(p\), it is enough that \(C_s\) omit the
antipode \(p^*\), and our radius choice supplies all the required segments.

\begin{proposition}\label{prop:level-hm}
The sublevel set \(C_s\) and the superlevel set \(D_s\) are homology \((n-1)\)-manifolds with boundary in the
sense of Mitchell, and their intrinsic boundary is precisely \(E_s\).
Consequently \(E_s\) is a compact cohomologically locally connected
integral homology \((n-2)\)-manifold without boundary in the sense of
Mitchell. Furthermore, \(E_s\) has the Borel--Moore homology of
\(S^{n-2}\).
\end{proposition}

\begin{proof}
We first consider \(C_s\). By Proposition~\ref{prop:subsuper-ar}, \(C_s\) is a compact
finite-dimensional metric AR.  In particular, it is locally compact
Hausdorff and first countable, and
\[
                 \dim_{\Z}C_s\leq\dim C_s\leq n-1.
\]
Mitchell uses Borel--Moore homology
\cite[p.~510]{Mitchell}.  Since \(C_s\) and
\(C_s\setminus\{z\}\) are locally contractible, the comparison
isomorphism of \cite[Corollary~V.13.6, p.~363]{Bredon} identifies
their singular relative homology with Borel--Moore relative homology
with compact supports.  We may therefore check Mitchell's local
homology condition using singular homology.  This comparison is used
for \(C_s\), not for \(E_s\).

At an interior point \(z\in C_s \setminus E_s\), excision identifies the local
homology of \(C_s\) with that of \(L\), hence with that of
\(\R^{n-1}\).  If \(z\in E_s\), both \(C_s\) and
\(C_s\setminus\{z\}\) are contractible by
Proposition~\ref{prop:subsuper-ar}; the ordinary long exact sequence of the
pair gives
\[
                    H_i(C_s,C_s\setminus\{z\};\Z)=0
                         \quad\text{for all }i.
\]
Thus, in every degree other than \(n-1\) the local group is zero, while
the group in degree \(n-1\) is either \(\Z\), at an interior point, or
zero, exactly at a point of \(E_s\).  This is precisely Mitchell's
definition of a homology \((n-1)\)-manifold with boundary over
\(\Z\), and its intrinsic boundary is \(E_s\).

Mitchell's theorem~\cite[Theorem, p.~510]{Mitchell} now says that this
boundary is a homology \((n-2)\)-manifold with empty boundary.  Since \(E_s\) is compact metrizable, it is first countable.
Applying \cite[Lemma~1, p.~510]{Mitchell} to \(E_s\) shows that
it is cohomologically locally connected.

For the global homology, we use \(C_s\) and its interior, as in the
proof of \cite[Proposition~3.9]{Thurston}.  In this calculation,
\(H_*^{\mathrm{BM}}\) denotes Borel--Moore homology with closed supports.
Set \(W=C_s\setminus E_s=\{h<s\}\) and \(q=n-1\).
By Proposition~\ref{prop:subsuper-ar}, both \(C_s\) and \(W\)
are contractible.  The open homology \(q\)-manifold \(W\) is therefore
orientable.  Excision for the closed subset \(E_s\subset C_s\)
and Poincar\'e duality give
\[
 H_k^{\mathrm{BM}}(C_s,E_s;\Z)
 \cong H_k^{\mathrm{BM}}(W;\Z)
 \cong H^{q-k}(W;\Z)
 \cong
 \begin{cases}
  \Z,& k=q,\\
  0,& k\ne q;
 \end{cases}
\]
see \cite[Corollary~V.5.10 and Theorem~V.9.2, pp.~312, 329]{Bredon}.
Here \(H^*\) denotes sheaf cohomology.  Since \(C_s\) is a compact
contractible ANR, the exact homology sequence of \((C_s,E_s)\) yields
\[
 H_k^{\mathrm{BM}}(E_s;\Z)\cong H_k(S^{n-2};\Z).
\]

The same argument applies to \(D_s\), with
\(D_s\setminus E_s=\{h>s\}\) in place of \(W\).  In particular,
\(D_s\) is also a homology \((n-1)\)-manifold with boundary \(E_s\),
and the global homology of \(E_s\) can equally be computed using
\(D_s\) and its interior.
\end{proof}

\begin{lemma}\label{lem:level-connected}
For every \(0<s<2r\), the level \(E_s\) is connected and locally path connected, and hence path connected.
\end{lemma}

\begin{proof}
We first prove connectedness.  By Proposition~\ref{prop:level-hm},
\(E_s\) is compact and cohomologically locally connected, and
\(H_0^{\mathrm{BM}}(E_s;\Z)\cong\Z\).
By \cite[Theorem~V.5.14, p.~314]{Bredon}, this group is the free
abelian group on the connected components of \(E_s\).
Thus \(E_s\) is connected.

We next prove local path connectedness.  The degree-zero part of
cohomological local connectedness implies local connectedness;
see \cite[Chapter~II, \S17, p.~126]{Bredon}.
Since \(E_s\) is a compact metric space, it is complete.
The Mazurkiewicz--Moore--Menger theorem
\cite[\S50, Theorem~1, p.~254]{KuratowskiTopology}
therefore implies that \(E_s\) is locally arcwise connected.
In particular, \(E_s\) is locally path connected and, being connected,
is path connected.
\end{proof}

\section{Arcs in a distance level}\label{sec:arcs-in-a-distance}

The aim of this section is to prove that an embedded arc contained in
an exact distance level \(E_s\) can be avoided by maps of \(I^2\) after
arbitrarily small homotopies in \(L\).  From now on assume \(n\geq5\).
There are two ingredients.  First, Bredon's nonseparation theorem
\cite[Theorem~V.16.20]{Bredon}, together with local path connectedness,
gives arcwise connectedness of the complement of such an arc inside
connected open subsets of \(E_s\).
Second, the level-changing motion \(T\) from
Lemma~\ref{lem:level-deformation} allows paths to move between the two
components \(\{h<s\}\) and \(\{h>s\}\) in a controlled way.  Together
these facts give local path and loop control in \(L\setminus A\) near
\(A\).  Banakh--Cauty--Karassev's local-negligibility criterion
\cite[Theorem~3.2(1)]{BCK} then turns this complement control into
the homotopical \(Z_2\)-property.

\begin{definition}\label{def:homotopical-z2}
Following Banakh--Cauty--Karassev~\cite{BCK}, let \(M\) be a metric
ANR and let \(A\subset M\) be closed.  We call \(A\) a
\emph{homotopical \(Z_2\)-set} in \(M\) if, for every open cover
\(\mathcal U\) of \(M\) and every map \(f\colon I^2\to M\), there is a map
\[
                     f'\colon I^2\longrightarrow M\setminus A
\]
which is \(\mathcal U\)-homotopic to \(f\), in the sense fixed in Section~\ref{sec:introduction}.
\end{definition}

\begin{lemma}\label{lem:level-avoid}
Let \(A\subset E_s\) be an embedded arc.  If \(U\subset E_s\) is a nonempty
connected relative open set, then
\(U\setminus A\) is nonempty and arcwise connected.  Consequently, for every
\(a\in E_s\) and every relative neighborhood \(N\) of \(a\),
there is a connected relative open neighborhood \(U\), with
\(a\in U\subset N\), such that \(U\setminus A\) is nonempty and arcwise connected.
\end{lemma}

\begin{proof}
Set \(m=n-2\).  By Proposition~\ref{prop:level-hm} and
Lemma~\ref{lem:level-connected}, \(E_s\) is connected and
cohomologically locally connected, has finite cohomological
dimension over \(\Z\), and satisfies the Borel--Moore local
homology condition of an \(m\)-manifold.  Thus it satisfies
condition~(d) of \cite[Theorem~V.16.8, p.~375]{Bredon}.
By that theorem, \(E_s\) is a cohomology \(m\)-manifold over
\(\Z\).

Let \(U\subset E_s\) be a nonempty connected relative open set.
Then \(U\) is also a cohomology \(m\)-manifold.  The set
\(A\cap U\) is closed in \(U\), and
\[
   \dim_{\Z}(A\cap U)\leq\dim(A\cap U)\leq1\leq m-2.
\]
Moreover, \(\dim_{\Z}U=m\) by
\cite[Corollary~V.16.18, p.~383]{Bredon}, so \(U\) cannot be
contained in \(A\).  Hence \(U\setminus A\) is nonempty, and
\cite[Theorem~V.16.20, p.~383]{Bredon} shows that it is connected.
It is open in the locally path-connected space \(E_s\), by
Lemma~\ref{lem:level-connected}, and is therefore path connected
and hence arcwise connected.

Finally, local connectedness of \(E_s\) gives, for every
\(a\in E_s\) and every relative neighborhood \(N\) of \(a\), a
connected relative open neighborhood \(U\) with
\(a\in U\subset N\).  The preceding argument applies to this \(U\).
\end{proof}

\begin{lemma}\label{lem:subsuper-null}
Let \(A\subset E_s\) be closed and let \(a\in A\).  For every
neighborhood \(U\) of \(a\) in \(L\), there is a neighborhood \(V\)
of \(a\) such that the inclusions
\[
 (V\cap C_s)\setminus A\hookrightarrow U\setminus A,\qquad
 (V\cap D_s)\setminus A\hookrightarrow U\setminus A
\]
are null-homotopic.  The first null-homotopy may be chosen in
\((U\cap C_s)\setminus A\), and the second in
\((U\cap D_s)\setminus A\).
\end{lemma}

\begin{proof}
The argument is almost identical to the local contractibility proof in
Proposition~\ref{prop:subsuper-ar}.  Here the homotopies are restricted to
\((V\cap C_s)\setminus A\) and \((V\cap D_s)\setminus A\) near \(a\in A\). We prove the assertion for \(C_s\); the proof for \(D_s\) is the same
with the sign reversed.  Choose \(\tau<0\) sufficiently close to zero so that the path \(u\mapsto T(a,u\tau)\), \(0\leq u\leq1\), is contained in \(U\).  Its endpoint \(a_-=T(a,\tau)\) satisfies
\(h(a_-)<s\).  Since \(\{h<s\}\) is open and locally contractible,
choose a neighborhood \(W_-\) of \(a_-\) whose inclusion in
\(U\cap\{h<s\}\) is null-homotopic.  By the joint continuity of
\(T\), after shrinking a neighborhood \(V\) of \(a\) we have
\[
 T(x,u\tau)\in U\cap C_s\quad(0\leq u\leq1),
 \qquad T(x,\tau)\in W_-
\]
for every \(x\in V\cap C_s\).  If \(x\notin A\), then this path starts outside \(A\); for every \(u>0\), the strict inequality
\eqref{eq:sign} gives \(h(T(x,u\tau))<s\), so the remainder of the
path also misses \(A\subset E_s\).  This deforms
\((V\cap C_s)\setminus A\) into \(W_-\), after which the chosen
null-homotopy contracts it inside \(U\setminus A\).
\end{proof}

\begin{lemma}\label{lem:crossing}
Let \(z\in E_s\).  For every neighborhood \(U\) of \(z\)
in \(L\), there is a neighborhood \(V\) of \(z\), with
\(V\subset U\), such that:
\begin{enumerate}[label=(\roman*)]
\item any two points of \(V\cap\{h<s\}\) can be joined by a path in
      \(U\cap\{h<s\}\);
\item any two points of \(V\cap\{h>s\}\) can be joined by a path in
      \(U\cap\{h>s\}\);
\item if \(x\in V\cap\{h<s\}\) and \(y\in V\cap\{h>s\}\), then
      \(x\) and \(y\) can be joined by a path in \(U\) whose
      intersection with \(E_s\) is exactly \(\{z\}\) at one parameter.
\end{enumerate}
\end{lemma}

\begin{proof}
Apply Lemma~\ref{lem:subsuper-null} with \(A=E_s\).
The last assertion of that lemma gives a neighborhood \(V\subset U\)
of \(z\) satisfying (i) and (ii).
Choose \(\tau>0\) so small that \(T(z,t)\in V\) for
\(-\tau\leq t\leq\tau\), and let
\(z_-=T(z,-\tau)\), \(z_+=T(z,\tau)\).
For (iii), join \(x\) to \(z_-\) in \(U\cap\{h<s\}\) by (i),
follow \(t\mapsto T(z,t)\) from \(z_-\) to \(z_+\), and join
\(z_+\) to \(y\) in \(U\cap\{h>s\}\) by (ii).
By \eqref{eq:sign}, the resulting path meets \(E_s\) at exactly
one parameter, with value \(z\).
\end{proof}

\begin{lemma}\label{lem:finite-crossings}
Let \(K\subset L\) be compact, and let \(U_0\subset L\) be an open
neighborhood of \(K\).  There is an open neighborhood \(N\) of \(K\),
with \(N\subset U_0\), such that every path \(\alpha\colon I\to N\)
is homotopic in \(U_0\) to a path \(\alpha'\) for which
\((\alpha')^{-1}(E_s)\) is finite and contained in \((0,1)\).
At each such parameter the path passes from \(\{h<s\}\) to \(\{h>s\}\),
or conversely, and its value belongs to \(K\cap E_s\).
For loops the construction is made cyclic.
\end{lemma}

\begin{proof}
For each \(q\in K\), use local contractibility of \(L\) to choose open
neighborhoods
\[
                         q\in V_q\subset W_q\subset U_0
\]
such that \(W_q\hookrightarrow U_0\) is null-homotopic.
If \(q\notin E_s\), choose \(W_q\) in the same open side of \(E_s\)
as \(q\), and choose \(V_q\) path connected.  If \(q\in E_s\), shrink
\(V_q\) so that Lemma~\ref{lem:crossing} applies with \(U=W_q\).
Choose finitely many \(V_q\)'s covering \(K\), and let
\(N=\bigcup_q V_q\).

Subdivide the parameter interval so that the image of each subinterval
under \(\alpha\) lies in a single \(V_q\).  At each subdivision vertex
whose image is in \(E_s\), move that image a small distance off \(E_s\)
using \(T\).  Each vertex track can be chosen in the intersection of the
\(V_q\)'s assigned to its incident subintervals. Thus the new endpoints of each subinterval
lie in the same \(V_q\) as its original subpath.

When \(q\notin E_s\), join the new endpoints in the path-connected set
\(V_q\).  When \(q\in E_s\), apply Lemma~\ref{lem:crossing} in \(W_q\).
If both endpoints lie on the same side, the replacement misses \(E_s\).
Otherwise use the crossing path constructed there, which meets \(E_s\)
at exactly one parameter, with value \(q\in K\cap E_s\).
For each subinterval, its original subpath, the replacement, and the two
vertex tracks form a loop in \(W_q\).  This loop is null-homotopic in
\(U_0\), so it extends over the homotopy square.  The extensions agree
on the prescribed vertex tracks and glue to the required homotopy.
For a loop, use the same track at the initial and terminal vertex.
\end{proof}

The point of the preceding lemma is that no regularity of the original
intersection \(\alpha^{-1}(E_s)\) is assumed.  It may, for example, be
infinite or contain a Cantor set.  The argument uses only the local
crossing lemma and local contractibility of \(L\).

\begin{proposition}\label{prop:local-complement}
Let \(A\subset E_s\) be an embedded arc.  For every \(a\in A\) and
every neighborhood \(U\) of \(a\) in \(L\), there is a neighborhood
\(V\) of \(a\) such that:
\begin{enumerate}[label=(\roman*)]
\item any two points of \(V\setminus A\) can be joined by a path in
      \(U\setminus A\);
\item every loop in \(V\setminus A\) is null-homotopic in
      \(U\setminus A\).
\end{enumerate}
\end{proposition}

\begin{proof}
Apply Lemma~\ref{lem:subsuper-null} to choose an open neighborhood \(W\) of
\(a\) such that both inclusions
\[
 (W\cap C_s)\setminus A\hookrightarrow U\setminus A,
 \qquad
 (W\cap D_s)\setminus A\hookrightarrow U\setminus A
\]
are null-homotopic.  By Lemma~\ref{lem:level-avoid}, choose a connected relative open
neighborhood \(W_s\) of \(a\) in \(E_s\), with
\(W_s\subset E_s\cap W\), such that \(W_s\setminus A\) is nonempty and arcwise connected.
Shrink \(V\) so that
\(\overline V\subset W\) and 
\(V\cap E_s\subset W_s\).

For (i), choose \(z\in W_s\setminus A\).  This point belongs to both
\((W\cap C_s)\setminus A\) and \((W\cap D_s)\setminus A\).
Each of the two null-homotopies above joins any two points of its domain
by a path in \(U\setminus A\).  Every point of \(V\setminus A\) belongs
to at least one of these domains, so it can be joined to \(z\) there.
This proves (i).

For (ii), let \(\gamma\) be a loop in \(V\setminus A\).  Apply
Lemma~\ref{lem:finite-crossings} to the compact set
\(K=\gamma(I)\subset W\setminus A\), with
\(U_0=W\setminus A\).  Since \(K\subset V\) and
\(V\cap E_s\subset W_s\), that lemma places every
crossing point in \(K\cap E_s\subset W_s\setminus A\).  Thus, after a
homotopy in \(W\setminus A\), we may assume that \(\gamma^{-1}(E_s)\) is finite, with all crossing values belonging to
\(W_s\setminus A\), and that at each intersection the path passes from \(\{h<s\}\) to \(\{h>s\}\), or conversely.  If there are no
intersections, the connected image of \(\gamma\) lies entirely in one
of \(\{h<s\}\) or \(\{h>s\}\), and the corresponding null-homotopy from Lemma~\ref{lem:subsuper-null} contracts it in \(U\setminus A\).

Otherwise, let \(\alpha\) be a maximal subpath of \(\gamma\) whose interior lies in \(\{h>s\}\), with endpoints \(x,y\in W_s\setminus A\).  By
Lemma~\ref{lem:level-avoid}, choose an arc
\(\beta\subset W_s\setminus A\) joining \(x\) to \(y\), taking instead
the constant path when \(x=y\).  The loop
\(\alpha\cdot\beta^{-1}\) lies in
\((W\cap D_s)\setminus A\), and is therefore null-homotopic in
\(U\setminus A\).  Thus \(\alpha\) may be replaced, relative to its
endpoints, by \(\beta\).  Repeating this for all such subpaths in \(\{h>s\}\) produces a loop contained in \((W\cap C_s)\setminus A\), which is
null-homotopic in \(U\setminus A\) by the null-homotopy for \(C_s\) from Lemma~\ref{lem:subsuper-null}.
\end{proof}

\begin{theorem}\label{thm:level-arc}
For \(0<s<2r\), every embedded arc \(A\subset E_s\) is
a homotopical \(Z_2\)-set in \(L\).
\end{theorem}

\begin{proof}
We use the characterization of homotopical \(Z_2\)-sets by local
negligibility.  Since \(A\) is an embedded arc, it is compact and
hence closed in the metric space \(L\).  By
Banakh--Cauty--Karassev~\cite[Theorem~3.2(1)]{BCK}, it is therefore
enough to prove that \(A\) is locally \(2\)-negligible in \(L\).

Recall that local \(2\)-negligibility means the following
\cite[pp.~36--37]{BCK}: for every \(x\in L\), every
neighborhood \(U\) of \(x\), and every \(k=0,1,2\), there is a
neighborhood \(V\subset U\) of \(x\) such that every map of pairs
\[
 f\colon (I^k,\partial I^k)
       \longrightarrow (V,V\setminus A)
\]
is homotopic, through maps of pairs into
\((U,U\setminus A)\), to a map whose image is contained in
\(U\setminus A\).

If \(x\notin A\), choose \(V\subset U\setminus A\), and there is
nothing to prove.  We therefore fix \(a\in A\) and a neighborhood
\(U\) of \(a\).

Since \(L\) is an ANR, it is locally contractible.  Choose an open
neighborhood \(W\) of \(a\) such that
\(a\in W\subset U
\)
and the inclusion \(W\hookrightarrow U\) is null-homotopic.
Apply Proposition~\ref{prop:local-complement} with \(W\) in place of
its target neighborhood.  We obtain a neighborhood \(V_0 \subset W\) of \(a\)
such that
\begin{enumerate}[label=(\alph*)]
\item any two points of \(V_0\setminus A\) can be joined by a path
      in \(W\setminus A\);
\item every loop in \(V_0\setminus A\) is null-homotopic in
      \(W\setminus A\).
\end{enumerate}
Because \(L\) is locally path connected, we may choose a
path-connected open neighborhood
\(a\in V\subset V_0\).
Moreover, \(V\setminus A\neq\varnothing\), since
\(L\setminus A\) is dense in \(L\): a point of
\(A\subset E_s\) can be moved arbitrarily little into either
\(\{h<s\}\) or \(\{h>s\}\) by
Lemma~\ref{lem:level-deformation}.

We verify the three cases \(k=0,1,2\).

For \(k=0\), a map \(I^0\to V\) is just a point of \(V\).
Since \(V\) is path connected and \(V\setminus A\neq\varnothing\),
that point can be joined by a path in \(V\subset U\) to a point of
\(V\setminus A\).  This gives the required homotopy.

For \(k=1\), let
\[
 f\colon (I,\partial I)\longrightarrow(V,V\setminus A).
\]
The two endpoints of \(f\) lie in \(V\setminus A\).  By (a), there
is a path
\[
                         g\colon I\longrightarrow W\setminus A
\]
with the same endpoints as \(f\).  The loop obtained by following
\(f\) and then \(g\) in reverse is contained in \(W\).  Since
\(W\hookrightarrow U\) is null-homotopic, this loop is
null-homotopic in \(U\).  Hence \(f\) is homotopic to \(g\),
relative to its endpoints, through maps of pairs into
\((U,U\setminus A)\).  Since \(g(I)\subset U\setminus A\), this
proves the required assertion for \(k=1\).

Finally, let
\[
 f\colon (I^2,\partial I^2)
       \longrightarrow(V,V\setminus A).
\]
Choose a base point \(*\in\partial I^2\) and put \(x_0=f(*)\).
The boundary map
\(
 \lambda=f|_{\partial I^2}
\)
is a loop in \(V\setminus A\), and hence, by (b), is
null-homotopic in \(W\setminus A\).

Consider the image of the relative homotopy class of \(f\) in
\(\pi_2(W,W\setminus A,x_0)\).
Its boundary under the exact homotopy sequence of the pair
\((W,W\setminus A)\) is the class of \(\lambda\), which is zero.
Exactness therefore shows that this relative class lies in the image
of
\[
                  \pi_2(W,x_0)
                  \longrightarrow
                  \pi_2(W,W\setminus A,x_0).
\]
Since the inclusion \(W\hookrightarrow U\) is null-homotopic, the
induced map on \(\pi_2\) is trivial (up to the usual change of base
point along the null-homotopy).  Naturality of the exact homotopy
sequences therefore implies that the image of the class of \(f\) in
\(\pi_2(U,U\setminus A,x_0)\)
is zero.  Equivalently, \(f\) is homotopic through maps of pairs into
\((U,U\setminus A)\) to a map whose image is contained in
\(U\setminus A\).

Thus \(A\) is locally \(2\)-negligible in \(L\).  Since \(L\) is
metric, hence Tychonoff, and \(A\) is closed,
Banakh--Cauty--Karassev~\cite[Theorem~3.2(1)]{BCK} imply that
\(A\) is a homotopical \(Z_2\)-set in \(L\).
\end{proof}

\section{From arcs in distance levels to DADP and DHP}

We now pass from the avoidance theorem for a single level arc to a global
general-position property of the link.  The idea is to approximate an
arbitrary path by finitely many level arcs and then use their homotopical
\(Z_2\)-property to move a disk off the resulting finite union.

\begin{lemma}\label{lem:common-level}
For \(p\in L\) and \(s>0\), write \(E_s(p)=\{z\in L:d(p,z)=s\}\).  If \(x\ne y\) lie in a path-connected set \(Q\subset L\), there are \(p\in Q\) and \(s>0\) such that \(x,y\in E_s(p)\).  If
\(\operatorname{diam}Q<\delta\), the level \(E_s(p)\) lies in
\(B(x,2\delta)\).
\end{lemma}

\begin{proof}
Choose a path \(\beta\) from \(x\) to \(y\) in \(Q\).  The continuous
function
\[
 t\longmapsto d(\beta(t),x)-d(\beta(t),y)
\]
changes sign, so at some \(t_0\), with \(p=\beta(t_0)\),
\[
                         d(p,x)=d(p,y)=s>0.
\]
Thus \(x,y\in E_s(p)\).  If \(\operatorname{diam}Q<\delta\), then
\(s<\delta\), and for \(z\in E_s(p)\),
\[
                         d(z,x)\leq d(z,p)+d(p,x)=2s<2\delta .
\]
\end{proof}

\begin{corollary}\label{cor:path-approx}
For every path \(\alpha\colon I\to L\) and every \(\varepsilon>0\),
there are a path \(\alpha'\colon I\to L\) and finitely many embedded
arcs
\[
                         A_1,\ldots,A_N\subset L
\]
such that
\[
                  \alpha'(I)\subset A_1\cup\cdots\cup A_N,
\]
each \(A_i\) is contained in a distance level
\[
                         E_{s_i}(p_i)
                         =\{z\in L:d(p_i,z)=s_i\},
                         \qquad 0<s_i<2r,
\]
and
\[
                  \sup_{t\in I}d\bigl(\alpha(t),\alpha'(t)\bigr)
                  <\varepsilon .
\]
Moreover, for every open cover \(\mathcal U\) of \(L\), the
approximation may be chosen so that \(\alpha'\) is
\(\mathcal U\)-homotopic to \(\alpha\).
\end{corollary}

\begin{proof}
Fix \(\varepsilon>0\), and choose
\(0<\delta<2r\)
with \(3\delta<\varepsilon\). By uniform continuity of \(\alpha\), choose a subdivision
\[
                  0=t_1<t_2<\cdots<t_{N+1}=1
\]
such that, for
\(
Q_i=\alpha([t_i,t_{i+1}])\),
one has
\(\operatorname{diam}Q_i<\delta
\)
for \(i=1,\ldots,N\).

Put
\(x_i=\alpha(t_i)\).
Suppose first that \(x_i\neq x_{i+1}\).  By
Lemma~\ref{lem:common-level}, there are \(p_i\in Q_i\) and
\(s_i>0\) such that
\(x_i,x_{i+1}\in E_{s_i}(p_i)\),
with
\(s_i<\delta<2r\)
and
\(E_{s_i}(p_i)\subset B(x_i,2\delta)\).
The point \(p\in L\) used in Sections~2--4 was arbitrary, so
Lemma~\ref{lem:level-connected} applies equally to
\(E_{s_i}(p_i)\).  Thus \(E_{s_i}(p_i)\) is path connected, and there is an
embedded arc
\(A_i\subset E_{s_i}(p_i)
\)
joining \(x_i\) to \(x_{i+1}\).  Parametrize \(A_i\) over
\([t_i,t_{i+1}]\), preserving the endpoints.

If \(x_i=x_{i+1}\), use the constant replacement path at \(x_i\).
To keep its image in a union of nondegenerate level arcs, choose
\(y_i\ne x_i\) in a small path-connected neighborhood of \(x_i\), and
apply the preceding construction to \(x_i,y_i\), obtaining
a level arc \(A_i\) containing \(x_i\).

Concatenating the replacement paths gives a path
\(\alpha'\colon I\to L\).  On a subinterval with distinct endpoints,
both the replacement arc and the original subpath lie near \(x_i\):
for \(t\in[t_i,t_{i+1}]\),
\[
 d\bigl(\alpha(t),\alpha'(t)\bigr)
 \leq d(\alpha(t),x_i)+d(x_i,\alpha'(t))
 <\delta+2\delta
 <\varepsilon .
\]
The same estimate is immediate on a constant replacement interval.
Hence \(\alpha'\) is \(\varepsilon\)-close to \(\alpha\), and its
image is contained in a finite union of embedded arcs lying in
distance levels \(E_s(p)\) with \(0<s<2r\).

Finally, \(L\) is a compact metric ANR.  Therefore, given an open cover
\(\mathcal U\) of \(L\), sufficiently close maps into \(L\) are
\(\mathcal U\)-homotopic.  Choosing the preceding approximation
sufficiently close proves the last assertion.
\end{proof}

\begin{lemma}\label{lem:finite-union}
A finite union of homotopical \(Z_2\)-sets in a metric ANR is a
homotopical \(Z_2\)-set.
\end{lemma}

\begin{proof}
It is enough to treat the union of two sets.  Let \(A_1,A_2\subset M\)
be homotopical \(Z_2\)-sets, let \(\mathcal U\) be an open cover of
the metric ANR \(M\), and let
\(f\colon I^2\to M
\)
be a map.

Choose an open star refinement \(\mathcal V\) of \(\mathcal U\).
Since \(A_1\) is a homotopical \(Z_2\)-set, there is a map
\(f_1\colon I^2\to M\setminus A_1
\),
which is \(\mathcal V\)-homotopic to \(f\).

The compact set
\(K=f_1(I^2)
\)
is contained in the open set \(M\setminus A_1\).  Choose an open
cover \(\mathcal W\) of \(M\), refining \(\mathcal V\), such that
every member of \(\mathcal W\) which meets \(K\) is contained in
\(M\setminus A_1\).  Since \(A_2\) is a homotopical \(Z_2\)-set,
there is a map
\(f_2\colon I^2\to M\setminus A_2
\),
which is \(\mathcal W\)-homotopic to \(f_1\).

For each \(x\in I^2\), the track of this second homotopy is contained
in some \(W_x\in\mathcal W\) containing \(f_1(x)\).  Hence
\(W_x\cap K\neq\varnothing\), and therefore
\(W_x\subset M\setminus A_1
\).
Thus the entire second homotopy, and in particular \(f_2(I^2)\),
misses \(A_1\).  Consequently
\(f_2(I^2)\cap(A_1\cup A_2)=\varnothing\).

It remains only to check the control.  The first homotopy track at
\(x\) lies in some \(V_x\in\mathcal V\), while the second lies in
some \(W_x\subset V'_x\) with \(V'_x\in\mathcal V\).  Both
\(V_x\) and \(V'_x\) contain \(f_1(x)\), and hence intersect.
Because \(\mathcal V\) star-refines \(\mathcal U\), their union is
contained in one member of \(\mathcal U\).  The concatenation of the
two homotopies is therefore a \(\mathcal U\)-homotopy from \(f\) to
\(f_2\).

Thus \(A_1\cup A_2\) is a homotopical \(Z_2\)-set.  The finite case
follows by induction.
\end{proof}

\begin{theorem}\label{thm:dadp-dhp}
The link \(L\) has the disjoint arc-disk property and the disjoint
homotopies property.
\end{theorem}

\begin{proof}
We first prove the homotopical
\(0\text{-}\DDh^{\{1,2\}}\)-property; see
Section~\ref{sec:introduction} for the definition.

Let \(\mathcal U\) be an open cover of \(L\), and let
\[
                  \alpha\colon I\longrightarrow L,
                  \qquad
                  f\colon I^2\longrightarrow L
\]
be maps.  By Corollary~\ref{cor:path-approx}, after a
\(\mathcal U\)-small homotopy we may replace \(\alpha\) by a path
\(\alpha'\) whose image is contained in a finite union
\[
                         A=A_1\cup\cdots\cup A_N
\]
of embedded arcs, where
\[
                         A_i\subset E_{s_i}(p_i),
                         \qquad 0<s_i<2r.
\]

The choice of the point \(p\in L\) in Sections~2--4 was arbitrary.
Hence Theorem~\ref{thm:level-arc} applies to each \(A_i\), and each
\(A_i\) is a homotopical \(Z_2\)-set in \(L\).  By
Lemma~\ref{lem:finite-union}, the finite union \(A\) is itself a
homotopical \(Z_2\)-set.

By the definition of a homotopical \(Z_2\)-set, \(f\) is
\(\mathcal U\)-homotopic to a map
\(f'\colon I^2\to L\setminus A\).
Since \(\alpha'(I)\subset A\), we have
\(\alpha'(I)\cap f'(I^2)=\varnothing\).
Thus \(L\) has the homotopical
\(0\text{-}\DDh^{\{1,2\}}\)
property in the notation of Banakh--Valov.  Since \(L\) is a metric
ANR, this is equivalent to the usual approximation formulation of
DADP used in Section~1.

DHP follows from \cite[Theorem~8.3(2)]{BV}, with \(m=0\) and \(n=k=1\):
\[
  0\text{-}\DDh^{\{1,2\}}
  \cap
  0\text{-}\DDh^{\{2,1\}}
  \ \Longrightarrow\
  1\text{-}\DDh^{\{1,1\}} .
\]
The two properties on the left differ only by interchanging the maps.
The property on the right is the homotopical form of DHP, and agrees with
its approximation form because \(L\) is a metric ANR.  As noted immediately
after that theorem, this implication for separable locally compact ANRs
was established earlier by Halverson~\cite[Corollary~4.5]{Halverson}.
\end{proof}

\section{The ambient disjoint disks property}

It remains to pass from \(L\times\mathbb R\) to the ambient \(G\)-space.
The punctured small metric ball already has a radial product structure.  The
only additional issue is the center, which we first show can be avoided by
maps of dimension at most two.

\begin{proposition}[avoidance of the center]\label{prop:vertex-z2}
The singleton \(\{c\}\) is a homotopical \(Z_2\)-set in \(X\).
\end{proposition}

\begin{proof}
For sufficiently small \(a>0\), put
\[
                         N_a=U(c,a),
        \qquad            L_a=S(c,a).
\]
As in \eqref{eq:radial-product}, the radial coordinates give
\[
        N_a\setminus\{c\}\cong L_a\times(0,a),
\]
while radial contraction to \(c\) shows that \(N_a\) is contractible.

The link \(L_a\) is path connected, and
Lemma~\ref{lem:link-pi1} gives
\(\pi_1(L_a)=0\).
Hence \(N_a\setminus\{c\}\) is path connected and simply connected.
The exact homotopy sequence of the pair
\((N_a,N_a\setminus\{c\})\) therefore gives
\[
        \pi_k\bigl(N_a,N_a\setminus\{c\}\bigr)=0,
        \qquad k=1,2.
\]
The corresponding \(k=0\) condition is immediate from the
path connectedness, and in particular nonemptiness, of
\(N_a\setminus\{c\}\).

The sufficiently small metric balls \(N_a\) form a neighborhood
basis at \(c\).  Hence the neighborhood-basis criterion in
\cite[Theorem~2.1(4)]{BCK} is satisfied through dimension \(2\), and
\(\{c\}\) is locally \(2\)-negligible at \(c\).

At every point \(x\neq c\), local \(2\)-negligibility is immediate:
one may choose a neighborhood of \(x\) disjoint from \(\{c\}\).
Hence \(\{c\}\) is locally \(2\)-negligible in \(X\).

Finally, \(\{c\}\) is closed and \(X\) is metric, hence Tychonoff.
Banakh--Cauty--Karassev~\cite[Theorem~3.2(1)]{BCK} therefore imply
that \(\{c\}\) is a homotopical \(Z_2\)-set in \(X\).
\end{proof}

\begin{proof}[Proof of Theorem~\ref{thm:main}]
By Theorem~\ref{thm:dadp-dhp}, the link \(L\) has DADP.
The compact metric ANR \(L\) is locally \(1\)-connected and has local
dimension \(n-1\geq4\), so it satisfies the standing hypotheses of
\cite[Section~2]{Daverman}.  Daverman's product theorem
\cite[Proposition~2.10]{Daverman} therefore gives
\[
                     \DADP(L)\Longrightarrow\DDP(L\times\R).
\tag{6.1}\label{eq:product}
\]
Alternatively, the same conclusion follows from the DHP assertion of
Theorem~\ref{thm:dadp-dhp} and Halverson's product theorem
\cite[Theorem~3.4]{Halverson}.

We first show that every sufficiently small metric ball in \(X\)
has DDP.  Fix \(c\in X\), choose \(r\) as in
\eqref{eq:radius-choice}, and let
\(N=U(c,2r)\). Let
\(f,g\colon I^2\to N
\)
be two maps, and let \(\varepsilon>0\).
Since \(f(I^2)\cup g(I^2)\) is compact in \(N\), there is
\(a<2r\) such that
\(f(I^2)\cup g(I^2)
                 \subset B(c,a)\).
Choose \(b\) with
\(a<b<2r\),
and then choose
\(0<\delta<
\min\left\{\frac{\varepsilon}{3},
\frac{b-a}{3}\right\}\). By Proposition~\ref{prop:vertex-z2}, the singleton \(\{c\}\) is a
homotopical \(Z_2\)-set in \(X\).  In particular, after choosing the
control sufficiently small, \(f\) and \(g\) admit
\(\delta\)-close approximations
\(f_1,g_1\colon I^2\longrightarrow X\setminus\{c\}\).
Because the original images are contained in
\(B(c,a)\) and \(\delta<b-a\), we have
\(f_1(I^2)\cup g_1(I^2)\subset U(c,b)\setminus\{c\}\). The compact set
\(f_1(I^2)\cup g_1(I^2)
\)
does not contain \(c\).  Hence there is \(\eta>0\) such that
\(d(c,x)>\eta
\)
for every
\(x\in f_1(I^2)\cup g_1(I^2)\).
Both images are therefore contained in the open annulus
\[
             A_{\eta,b}
             =\{x\in X:\eta<d(c,x)<b\}.
\]

By the radial product coordinates,
\[
                         A_{\eta,b}
                         \cong L\times(\eta,b)
                         \cong L\times\R.
\]
Equation~\eqref{eq:product} therefore implies that
\(A_{\eta,b}\) has DDP.  Applying DDP there with approximation
size less than \(\delta\), we obtain maps
\(f_2,g_2\colon I^2\longrightarrow A_{\eta,b}
\)
such that
\(f_2(I^2)\cap g_2(I^2)=\varnothing
\)
and
\[
        d(f_1,f_2)<\delta,\qquad
        d(g_1,g_2)<\delta
\]
uniformly on \(I^2\).  Consequently
\[
        d(f,f_2)<2\delta<\varepsilon,\qquad
        d(g,g_2)<2\delta<\varepsilon.
\]
Thus \(N\) has the usual disjoint disks property.

Since \(N\) is an open subset of the metric ANR \(X\), it is itself
a metric ANR.  By Hanner's controlled homotopy theorem
\cite[Theorem~4.1]{Hanner}, the usual approximation formulation of DDP on
\(N\) is equivalent to the homotopical
\(0\text{-}\DDh^{\{2,2\}}\)-property.  Hence every ball \(U(c,2r)\) above
has the homotopical
\(0\text{-}\DDh^{\{2,2\}}
\)
property.

The balls \(U(c,2r)\), with \(c\in X\) arbitrary and \(r\) chosen
as in \eqref{eq:radius-choice}, form an open cover of \(X\).
Since \(X\) is metric, it is paracompact and submetrizable.
Banakh--Valov~\cite[Proposition~5.4(2)]{BV} therefore imply that
\(X\) has the homotopical
\(0\text{-}\DDh^{\{2,2\}}
\)
property.

Finally, let \(\varepsilon>0\).  Apply this homotopical property to
an open cover of \(X\) by sets of diameter less than
\(\varepsilon\).  The resulting disk maps have disjoint images and
are \(\varepsilon\)-close to the original maps.  Hence \(X\) has
DDP.
\end{proof}

\begin{remark}
Proposition~\ref{prop:vertex-z2} can be bypassed as follows.  Indeed, by \cite[Lemma~3.2]{BHR},
\(\rho\) is locally bounded away from zero.  For each \(x\in X\),
choose \(r>0\) such that \(8r<\rho(c)\) for every \(c\in U(x,r)\),
and choose \(c\in U(x,r)\setminus\{x\}\).  Then
\(x\in U(c,4r)\setminus\{c\}\), and
\eqref{eq:radial-product} gives
\[
 U(c,4r)\setminus\{c\}
 \cong S(c,r)\times(0,4r)
 \cong S(c,r)\times\R.
\]
By Theorem~\ref{thm:dadp-dhp} and \eqref{eq:product}, this punctured
ball has DDP.  Since \(x\) was arbitrary, these punctured balls
form an open cover of \(X\).
Each is an ANR, so it also has the homotopical
\(0\text{-}\DDh^{\{2,2\}}\)-property.
Applying \cite[Proposition~5.4(2)]{BV} to this cover gives DDP
for \(X\).
\end{remark}

\section*{Acknowledgements}
The second author gratefully acknowledges the profound influence
of Bob Daverman on his mathematical development.
Daverman introduced him to decomposition theory and wild topology, and they worked
closely for many years, coauthoring several papers.

The authors used OpenAI Codex with the GPT-5.6 Sol Ultra model during
the development and preparation of this manuscript.  The authors first
developed the main mathematical ideas and an initial proof sketch, and
provided Codex with a collection of relevant references known to them
at that stage.  Codex was then used to help expand the sketch into a
preliminary manuscript and subsequently as an interactive research tool
for exploring and refining proof strategies, examining intermediate
arguments, searching for relevant literature, and improving the
exposition.  In particular, Codex identified the paper of Todorov and
Valov~\cite{TV} as potentially relevant to the arc-avoidance argument;
the proof presented here instead uses Bredon's nonseparation theorem.
The resulting
manuscript was checked, substantially revised, and rewritten by the
authors.  All mathematical statements, proofs, and references in the
final manuscript were independently verified by the authors, who take
full responsibility for the content.

The first author was supported by JSPS KAKENHI Grant
No.~25K23336.  The second author was supported by NSFC Grant
No.~12201102.

\end{document}